\documentclass[12pt]{amsart}

\usepackage{amssymb}

\usepackage{enumerate}
\usepackage{dsfont}

\makeatletter
\@namedef{subjclassname@2010}{%
  \textup{2010} Mathematics Subject Classification}
\makeatother

\newtheorem{thm}{Theorem}[section]

\newtheorem{lem}[thm]{Lemma}
\newtheorem{prop}[thm]{Proposition}

\newtheorem{mainthm}[thm]{Main Theorem}

\theoremstyle{definition}

\numberwithin{equation}{section}

\DeclareSymbolFont{bbold}{U}{bbold}{m}{n}
\DeclareSymbolFontAlphabet{\mathbbold}{bbold}
\renewcommand{\mod}[1]{\ \left(\textnormal{mod}\ #1\right)}
\renewcommand{\l}{\lambda}
\DeclareMathOperator*{\Osum}{\sum{}^*}

\begin{document}


\baselineskip=17pt



\title{Uniformly Counting  Rational Points on Conics}
\author[E.Sofos]{Efthymios Sofos}
\address{School of Mathematics \\ University of Bristol \\ Bristol, BS8 1TW, United Kingdom}
\email{efthymios.sofos@bristol.ac.uk}

\date{}

\begin{abstract}We provide an asymptotic estimate for the number of
rational points of bounded height on a 
non--singular conic
over $\mathbb{Q}$.
The estimate is uniform in the coefficients of the underlying quadratic form.
\end{abstract}



\maketitle

\section{Introduction}
\label{1s}
Let $Q\left(\mathbf{x}\right) \in \mathbb{Z}[x_1,x_2,x_3]$
be a non--singular quadratic form.
We denote by $\mathbb{Z}^3_{\text{prim}}$
the integer vectors $\mathbf{x}$
that are primitive, i.e. that satisfy
$\gcd(\mathbf{x})=1.$
Our main concern in this paper
regards the number of primitive integer zeros of $Q$
contained on an expanding region of $\mathbb{R}^3.$
It is therefore
only the case that $Q$ 
is isotropic
that we are interested in and we will proceed under this assumption for the rest of the paper.

For any arbitrary norm
$\|.\| : \mathbb{R}^3 \to \mathbb{R}_{\geq 0}$
define the counting function
$$
N\left(Q,B\right):=\#\{\mathbf{x} \in \mathbb{Z}^3_{\text{prim}}:Q\left(\mathbf{x}\right)=0,\|\mathbf{x}\|\leq B\}.$$
A very special case of the work~\cite{fmt} establishes
the asymptotic formula 
$$
N\left(Q,B\right) \sim c_{Q}   B,
$$
valid for $B \to \infty.$
This confirms the Manin conjecture and furthermore 
$c_{Q}=c_{Q}\left(\|.\|\right)$ is the constant predicted in~\cite{peyre}.

Let $\langle Q \rangle$ denote the 
maximum modulus of the coefficients of $Q.$ 
As pointed out in~\cite{tbvv},
one expects the existence of absolute constants $\beta,\gamma>0$ 
such that
\[
N\left(Q,B\right)=c_{Q}  B+O\left(B^{1-\gamma} {\langle Q \rangle}^{\beta}\right).
\] 
Our aim is
to establish such an estimate
and furthermore to state
explicitly  admissible values for
$\beta$ and $\gamma.$ 

We begin by recalling existing results related to this subject. 
Let 
\newline
$w:\mathbb{R}^3 \to \mathbb{R}_{\geq 0}$ be a smooth weight function of compact support
and~let $$N_w\left(Q,B\right):=\sum_{\substack{ \mathbf{x} \in \mathbb{Z}^3_{\text{prim}}\\
Q\left(\mathbf{x}\right)=0}} w\left(B^{-1}\mathbf{x}\right).$$
It is proved in~\cite[Cor.2]{HB'} that there exists a positive constant $c_1$
such that one has 
$$
N_{w}\left(Q,B\right)=c_{Q,w}  B+O_{Q,w}\left(B\exp\{-c_1\sqrt{\log B}\}\right),
$$
as $B \to \infty.$ The proof is carried out via a modification of the circle method.

Let $\Delta_Q$ and
$\delta_Q$ be the discriminant and the greatest common divisor of the $2 \times 2$ minors of the
matrix of the form $Q$ respectively. 
In~\cite[Cor. 2]{n-2}, it is proved
that 
$$
N\left(Q,B\right) \ll  
\tau\left(|\Delta_Q|\right)
\left(1+\frac{B{\delta_Q}^{1/2}}{|\Delta_Q|^{1/3}}\right),
$$
where $\tau$ denotes the divisor function.
It should be stressed that the implied constant 
is absolute.

We provide the definition of the leading constant 
$c_{Q}$ before stating our main result.
We define the Hardy--Littlewood local densities following~\cite{HB'}.
Let 
\begin{equation}
\label{si}
\sigma_\infty:=\sigma_{\infty}\left(Q,\|.\|\right)=\lim_{\epsilon \to 0}
\frac{1}{2\epsilon}
\int_{\substack{|Q\left(\mathbf{x}\right)|\leq \epsilon \\
\|\mathbf{x}\| \leq 1}}1 \,\mathrm{d}\mathbf{x},
\end{equation}
and similarly for any prime $p,$ let
\begin{equation}
\label{s}
\sigma_p:=\sigma_p\left(Q\right)=\lim_{n \to \infty}\frac{1}{p^{2n}}N_Q^*\left(p^n\right),
\end{equation}
where for any positive integer $n,$ 
$$N_Q^*\left(p^n\right):=
\#\{\mathbf{x} \mod{p^n}:p\nmid \mathbf{x},~Q\left(\mathbf{x}\right)\equiv 0 \mod{p^n}\}.
$$
The Peyre constant is then defined as 
\[
c_{Q}=
\frac{1}{2}\sigma_{\infty} 
\prod_{p}\sigma_p
\]
where
the product 
is taken over the set of primes
and is convergent.
Let
$C \subseteq \mathbb{P}^2$
be the smooth projective curve defined by $Q.$
The existence of the factor
$\frac{1}{2}$
is due to the fact 
that the anticanonical line bundle is twice the generator of the Picard group
$\mathrm{Pic}\left(C\right) \cong \mathbb{Z},$
where 
$\alpha\left(C\right)$ is 
the volume of a certain polytope 
contained in the cone of effective 
divisors.

Next, let
\begin{equation}
\label{psy}
K_0:=1+\sup_{\mathbf{x} \neq \mathbf{0}}
\frac{\ \ \|\mathbf{x}\|_{\infty}}{\|\mathbf{x}\|},
\end{equation}
and 
notice that $K_0$ is a constant depending only on the choice of norm $\|.\|.$
A norm 
$\|.\| : \mathbb{R}^3 \to \mathbb{R}_{\geq 0}$
is called isometric to the supremum norm 
$\|.\|_\infty$
when there exists an invertible matrix
$\mathfrak{g}
\in GL_3(\mathbb{R})$
such that 
$ \|\mathbf{x}\|
=\|\mathfrak{g}\mathbf{x}\|_{\infty}$
for all $\mathbf{x} \in \mathbb{R}^3.$

We have the following result.
\begin{mainthm} 
\label{thrm01}
Let $Q$ be a ternary non--singular integer quadratic form with a rational zero
and let $\|.\|$ be any norm isometric to the maximum norm.
Then
$$N\left(Q,B\right)= c_{Q} 
B+
O\left(\left(B K_0\right)^{\frac{1}{2}} \left(\log B K_0\right) \  {\langle Q \rangle}^{5}\right),$$
for $B \geq 2.$ 
The implied constant in the estimate
is absolute.  
\end{mainthm} 
The proof of Theorem~\ref{thrm01} reveals that
for any $\epsilon>0$,
at the expense of an implied constant
that depends on $\epsilon,$
one can replace the term ${\langle Q \rangle}^{5}$ appearing in the error term by 
${\langle Q \rangle}^{\frac{19}{20}+\epsilon}$
as well as 
${\langle Q \rangle}^{4+\epsilon}
\delta_Q^{\frac{1}{2}}$ ~(see~\eqref{lastend}).
Further improvements
may follow using~\cite[Theorem 1]{hool}.
We hope it will be
apparent to the reader that the main value of Theorem~\ref{thrm01} lies in its generality rather than the exponent of 
${\langle Q \rangle}$ obtained.

The proof is conducted in two stages. Firstly, in \S\ref{2s}--\S\ref{5s},
we prove Theorem~\ref{thrm01} for conics of a special shape, using the fact that since
$C\left(\mathbb{Q}\right)\!\neq \!\emptyset,$
there is a morphism $\mathbb{P}^1 \to C.$
The conditions involving the resulting 
parametrising functions
lead to a
lattice counting problem.
One should comment that
the choice of the parametrising functions 
is not unique and that choosing them 
appropriately plays a significant 
r\^ole. An amount of work
regarding this issue
has taken place, as the papers~\cite{crem}
and~\cite{simon}
reveal. The second stage is performed in \S\ref{7s}. Here we apply a unimodular transformation
to a conic of general shape to transform the problem into the one we have already treated. 

{\bf Notation.} 
The implied constants in the $O\left(.\right)$ notation will be absolute throughout this paper, 
except where specifically indicated, via the use of
a subscript.
The norm notation $\|.\|$ will be reserved for 
norms of elements of $\mathbb{R}^3$
while $\|.\|_{\infty}$
will be used for the matrix supremum
norm in $\mathbb{R}^{3\times 3}$,
defined by
$\|(a_{i,j})_{1\leq i,j \leq 3}\|_{\infty}:=\max_{1\leq i,j \leq 3}|a_{i,j}|,$
as well as the supremum norm of $\mathbb{R}^3.$ 
We denote the generalised divisor function by
$\tau_k\left(n\right),$ which is defined to be the number of representations of $n$ as the product 
of $k$ natural numbers. The well--known bound
$\tau_k\left(n\right) \ll_{k,\epsilon} n^\epsilon,$
valid for each $\epsilon>0,$ shall be used.
By $\Osum_{\left(s,t\right)\mod n},$ we shall mean a summation for
$s,t \in [1,n],$ subject to the condition
$\gcd(s,t,n)=1.$

\section{Preliminary estimates}
\label{2s}
Throughout \S\ref{2s}--\S\ref{5s}, we denote by $Q$ the quadratic forms 
of which $(0,1,0)$ is a zero, i.e.
$$
Q\left(\mathbf{x}\right)=ax^2+bxy+dxz+eyz+fz^2,
$$
where $a,\ldots,f \in \mathbb{Z}.$
We will denote 
by $\Delta_Q$
its
discriminant,
$$
\Delta_Q =ae^2
-deb+fb^2.
$$
It is our intention in the aforementioned sections
to prove the following 
special version of Theorem~\ref{thrm01}.
Its proof hinges upon the classical parametrisation of a conic by the lines
going through
a given point.
\begin{prop}
\label{prop1}
Let $Q$ be a non--singular integer ternary quadratic form as above. Then for any norm isometric to the maximum norm and for any $\epsilon>0,$ one has
$$
N\left(Q,B\right)=c_{Q}  B  
+O_{\epsilon}\!\left( \,
\left(B K_0\right)^{\frac{1}{2}}
\log \left(B K_0\right) \!
\min\Big\{|\Delta_Q|^{\frac{1}{4}},
\delta_Q^{\frac{1}{2}}\Big\}
\!
\left(|\Delta_Q|
+
{\langle Q \rangle}\right)
{\langle Q \rangle}^{\epsilon}
\, \right)\!,
$$
for $B\geq 2.$ 
\end{prop}
Let $\Pi$ be the matrix
$$
\Pi:=
 \begin{pmatrix}
 b   & e  &  0\\
  -a  & -d  &  -f  \\
 0 &  b  &  e
 \end{pmatrix}
$$
and define the three
binary quadratic forms $q_1,q_2,q_3$
such that
\begin{equation}
\label{nib}
\mathbf{q}\left(s,t\right)=
\Pi
\begin{pmatrix}
  s^2 \\
  st \\
  t^2
 \end{pmatrix}
\end{equation}
where $\mathbf{q}=\left(q_1,q_2,q_3\right)^{T}.$
One can verify that  
$\text{Det}\left(\Pi\right) = \Delta_Q$
and that in particular the matrix $\Pi$
is invertible. Hence one gets
\begin{equation}
\label{adj}
\mathrm{adj}\left(\Pi\right) \mathbf{q}\left(s,t\right)= \Delta_Q
\begin{pmatrix}
  s^2 \\
  s t \\
  t^2
 \end{pmatrix}.
\end{equation}
Notice that for
\begin{equation}
\label{el}
\begin{aligned}
g\left(s,t\right)&:=as^2+dst+ft^2, \\
L\left(s,t\right)&:=b s+e t, 
\end{aligned}
\end{equation}
one has 
\begin{equation}
\label{g}
\begin{aligned}
q_1\left(s,t\right)&=s L\left(s,t\right), \\
q_2\left(s,t\right)&=-g\left(s,t\right), \\
q_3\left(s,t\right)&=t L\left(s,t\right).
\end{aligned}
\end{equation}
For each integer $n,$ let
\begin{equation}
\label{ro}
\rho^*\left(n\right):=\#\{\left(s,t\right) \in 
[0,n)^2
:n|\mathbf{q}\left(s,t\right),~\gcd(s,t,n)=1\},
\end{equation}
and note that $\rho^*$ is a multiplicative function.
Equations \eqref{g} imply that
this 
expression equals
$$\rho^*\left(n\right)=\#\{\left(s,t\right) \in [0,n)^2
:n|\left(L\left(s,t\right),g\left(s,t\right)\right),~\gcd(s,t,n)=1\}.$$
\begin{lem}
\label{lem1}
\begin{enumerate}[\upshape (i)]
\item The function $\rho^*$ 
is supported on the 
divisors of $\frac{\Delta_Q}{\gcd(b,e)} .$ \label{lemmon}
\item For all integers $n$
we have 
$$
\rho^*\left(n\right)\leq n \gcd(b,e).$$
\end{enumerate}
\end{lem}
\begin{proof} (i) 
It suffices to show that 
for each prime $p$ and integer $\nu\geq 1$
with $\rho^*(p^\nu)\neq 0$ we have that 
$$\nu+\min\{v_p(b),v_p(e)\}\leq v_p(\Delta_Q).$$
Let $(s,t)$ be counted by
$\rho^*(p^\nu)$.
We may assume without loss of generality that $v_p\left(b\right) \leq v_p\left(e\right).$ 
Since $\gcd(b,e)^2|\Delta_Q$ our claim in the case
$\nu\leq v_p(b)$ 
is trivial.
If $\nu>v_p(b)$ then
we may write
$b=p^{v_p\left(b\right)}b',e=p^{v_p\left(e\right)}e'$ 
with $p\nmid b'e'.$
Plugging these values
in the congruence $L\left(s,t\right)\equiv 0 \mod{p^{\nu}} $
yields 
\begin{equation}
\label{something}
b's \equiv
-p^{v_p\left(e\right)-v_p\left(b\right)}
e't \mod{p^{\nu-v_p\left(b\right)}}
\end{equation}
and hence 
$p\nmid t$
since otherwise we would have
$p|(s,t)$
which would contradict the definition
of $\rho^*\left(p^n\right)$. 
We deduce that 
$$t^2\left(
ae^2
p^{-2v_p\left(b\right)}
-de
b'
p^{-v_p\left(b\right)}
+f{b'}^2\right)
\equiv
{b'}^2
g(s,t) \equiv
0
\mod{p^{\nu-v_p\left(b\right)}}
$$
and therefore 
$p^{\nu+v_p(b)}
|ae^2
-deb+fb^2
=\Delta_Q
$
which concludes the proof of the first part.

(ii) It suffices to prove that
for all primes $p$
and integers $\nu\geq 1$
we have 
\begin{equation}
\label{aux}
\frac{\rho^*\left(p^\nu\right)}{p^{\nu}} \leq 
p^{\min\{v_p\left(b\right),v_p\left(e\right)\}}.
\end{equation}
Let $\left(s,t\right)$ be counted by $\rho^*\left(p^\nu\right).$
We may assume as previously that we have
$v_p\left(b\right) \leq v_p\left(e\right)$. 
In the case that $\nu\leq v_p\left(b\right)$, then \eqref{aux} is a consequence of the trivial bound $\rho^*\left(p^\nu\right)\leq p^{2\nu}$.
In the opposite case
we proceed as in the proof of part~(i).
Then equation~\eqref{something}
shows that
the value of $ s / t \mod{p^{\nu-v_p\left(b\right)}}$ is uniquely determined 
and can be lifted to at most $p^{v_p\left(b\right)}$ values 
$\mod{p^\nu},$ which proves \eqref{aux} in all cases.
\end{proof}
We record a generalisation of
M\"{o}bius inversion that will be used later.
\begin{lem}
\label{lem2} 
Let $\mathcal{A}$ be a finite subset of $\mathbb{Z}^2$
and $n$ a fixed integer. Then 
\begin{align*}  &\#\{\left(s,t\right) \in \mathcal{A}:\gcd(s,t)=1\}\\
&=\sum_{\substack{ m=1 \\ \gcd(m,n)=1}}^{\infty}\mu\left(m\right)
\#
\left\{\left(s,t\right) \in \mathcal{A}:
\begin{array}{l}
\gcd(s,t,n)=1,\\
m|s,m|t
\end{array}
\right\}.
\end{align*} 
\end{lem}
\begin{proof}
Define $\mathds{1}_{\mathcal{A}}:\mathbb{Z}^2 \to \{0,1\}$ as the 
indicator function of $\mathcal{A}.$ 
M\"obius inversion gives
$$\sum_{\substack{\gcd(s,t,n)=1\\ \gcd(s,t)=1}}
\mathds{1}_{\mathcal{A}}\left(s,t\right)=
\sum_{m=1}^{\infty} \mu\left(m\right)
\sum_{\substack{\gcd(s,t,n)=1\\ m|s, m|t}}
\mathds{1}_{\mathcal{A}}(s,t).
$$
Our assertion is proved upon
noticing that only $m$ coprime to $n$
are taken into account in the summation.
\end{proof}

\section{Parametrisation of the conic}
\label{3s}
In this section, we begin by showing how
the problem of counting points on conics can be rephrased 
using the parametrisation functions $\mathbf{q}\left(s,t\right).$
This will lead us to count primitive integer points in regions
of $\mathbb{R}^2.$ 

Let
\begin{equation}
\label{cal}\mathcal{N}\left(Q,B\right) :=\#
\left\{\left(s,t\right) \in \mathbb{Z}_{\text{prim}}^2:t>0,
 \|\mathbf{q}\left(s,t\right)\| \leq \lambda B
\right\},
\end{equation}
where $\l=\gcd(\mathbf{q}\left(s,t\right)) \in \mathbb{Z}.$
\begin{lem}
\label{lem3} 
One has
$N\left(Q,B\right)=\mathcal{N}\left(Q,B\right)+O\left(1\right),$
where the implied constant is absolute.
\end{lem}
\begin{proof} 
Let $C\subset \mathbb{P}^2$ be the curve given by $Q=0$
and denote the point $\left(0,1,0\right)$ of $C$ by
$\xi.$ The tangent line to $C$ through
$\xi,$
is given by
$$L_{\xi}:=\{
ez=bx\}
.$$
Let $\mathcal{L}$ be the set of projective lines 
in $\mathbb{P}^2$ that pass through
$\xi$ and 
$\mathcal{L}\left(\mathbb{Q}\right)$ be the corresponding 
subset of lines that are defined over $\mathbb{Q}.$
Define $U \subset C$ as the open subset formed by deleting
$\xi$ from $C.$ Letting $\mathcal{U}:=
\mathcal{L}\setminus \{L_{\xi}\},$ we note that
the sets $U\left(\mathbb{Q}\right)$ and $\mathcal{U}\left(\mathbb{Q}\right)$
are in bijection.

The general element of $\mathcal{L}\left(\mathbb{Q}\right)$
is given by
$$
L_{s,t}:=\{sz=tx\}
$$ 
for integer pairs $\left(s,t\right)$
such that 
$\gcd(s,t)=1.$ 
The condition
$\left(s,t\right) \neq \frac{\left(b,e\right)}{\gcd(b,e)}$
ensures that 
we have a point in $\mathcal{U}\left(\mathbb{Q}\right).$
One can ignore this, since
the contribution of 
such $s,t$ is $O\left(1\right).$ The bijection between lines with 
$t>0$ and $t<0$ allows us to consider the contribution coming from the former. The contribution of pairs $\left(s,t\right)$ with $t=0$ 
is $O\left(1\right)$ due to the condition $\gcd(s,t)=1$.

One can make explicit the bijection between
$U\left(\mathbb{Q}\right)$
and 
$\mathcal{U}\left(\mathbb{Q}\right)$
as follows.
Recall the definition of
$L,g$ in \eqref{el}.
A computation reveals that
the line $L_{s,t}$ intersects $C$
in the point $(x,y,z)$
if and only we have
 $zg\left(s,t\right)+ytL\left(s,t\right)=0$
or
$z=0$ holds.
In the latter case, one gets the point
$\xi,$
which is to be ignored.
In the former case, we have
$$-g\left(s,t\right)xt=-g\left(s,t\right)sz=syL\left(s,t\right)t,$$
by the equation for $L_{s,t}.$ 
The primitive integer vectors $\left(x,y,z\right)$
represent a point 
in $C\left(\mathbb{Q}\right)$ 
if and only if
$$\left(x,y,z\right)=\pm\left(
sL\left(s,t\right)/\l,
-g\left(s,t\right)/\l,
tL\left(s,t\right)/\l
\right),$$
where $\l=\gcd(sL\left(s,t\right),-g\left(s,t\right),tL\left(s,t\right)).$
Making use of \eqref{g} concludes the proof of the lemma.
\end{proof}

Let us define
for any 
$T \in \mathbb{R}_{\geq 1}$ and $n,\sigma,\tau \in \mathbb{N},$
\begin{equation}
\label{0star}
M^*_{\sigma,\tau}\left(T,n\right):=
\#
\left\{\left(s,t\right) \in \mathbb{Z}_{\text{prim}}^2:
\begin{array}{l}\left(s,t\right) 
\equiv \left(\sigma,\tau\right) \mod n,\\
t>0,~\|\mathbf{q}\left(s,t\right)\|\leq T
\end{array}
\right\}.
\end{equation}
\begin{lem}
\label{lem4}
One has$$
\mathcal{N}\left(Q,B\right)=
\sum_{k\l 
|
\Delta_Q / \gcd(b,e)} 
\mu\left(k\right)
\Osum_{\substack{\left(\sigma,\tau\right) \mod{k\l} \\ k\l|\left(L\left(\sigma,\tau\right),g\left(\sigma,\tau\right)\right) }}
\hspace{0.0001cm}
M^*_{\sigma,\tau}\left(B\l,k\l\right).
$$
\end{lem}
\begin{proof} 
Any integer $\lambda$ that appears in \eqref{cal}, satisfies $\lambda|\mathbf{q}\left(s,t\right)$ 
for some coprime integers $s,t$, so part $\left(i\right)$ of
Lemma~\ref{lem1} implies that 
$\l| \frac{\Delta_Q}{\gcd(b,e)}.$ We therefore
get
\begin{align*}
\mathcal{N}\left(Q,B\right)&=
\sum_{\l
|
\Delta_Q / \gcd(b,e)
}
\#
\left\{\left(s,t\right) \in \mathbb{Z}^2_{\text{prim}}:
\begin{array}{l}
\l | \mathbf{q}\left(s,t\right),~\gcd(\frac{\mathbf{q}\left(s,t\right)}{\lambda})=1,\\
t>0,~\|\mathbf{q}\left(s,t\right)\|\leq B \lambda
\end{array}
\right\}.
\end{align*}
Using Lemma~\ref{lem2} with $n=1,$ gives
\begin{equation}
\label{ena}
\mathcal{N}\left(Q,B\right)=\sum_{k \lambda| 
\Delta_Q / \gcd(b,e)}  \mu\left(k\right) M^{\ast}\left(B\lambda,k\lambda\right),
\end{equation}
where for any $T\geq 1, n \in \mathbb{N},$ 
we have defined
$$
M^{\ast}\left(T,n\right):=
\#
\left\{\left(s,t\right) \in \mathbb{Z}^2_{\text{prim}}:
\begin{array}{l}
 n|\mathbf{q}\left(s,t\right),~t>0, \\
\|\mathbf{q}\left(s,t\right)\|\leq T
\end{array}
\right\}.
$$  Partitioning into congruence classes$\mod n$ yields
\[
M^{\ast}\left(T,n\right)=\Osum_{\substack{\left(\sigma,\tau\right) \mod{n} \\ n|\left(L\left(\sigma,\tau\right),g\left(\sigma,\tau\right)\right)}}
\hspace{0.0001cm}
M^*_{\sigma,\tau}\left(T,n\right),
\]
which, when used along with \eqref{ena},
yields the proof of the lemma.
\end{proof}

\section{Counting lattice points}
\label{4s}
The quantity appearing in \eqref{0star}
involves integer points $\left(s,t\right)$ which are primitive. 
We will use M\"{o}bius inversion to deal with this condition. This will
lead us to count
integer points in a dilated
region.
In order to do so, one needs 
certain information regarding this region, 
which is the purpose of the next lemma.

Recall the definition \eqref{nib}. 
Denote by $V$ the region
\begin{equation}
\label{psy00}
V:=\{\left(s,t\right) \in \mathbb{R}^2:t>0, \|\mathbf{q}\left(s,t\right)\| \leq 1\}.
\end{equation}

\begin{lem}
\label{lem5} 
$V$ is bounded and in particular, 
it is contained in the rectangle given by 
$$
|s|,|t| \ll {\langle Q \rangle} \left(\frac{K_0}{|\Delta_Q|}\right)^{\frac{1}{2}}.
$$
The length of the boundary 
of $V$, denoted by $|\partial V|$,
satisfies  
$$
|\partial V|
\ll {\langle Q \rangle} \left(\frac{K_0}{|\Delta_Q|}\right)^{\frac{1}{2}}
\! ,
$$
where the implied constant is absolute.
Furthermore
any line parallel to one of the $2$ coordinate axes intersects $V$ in a set of points which, if not empty, consists of at most $O(1)$ intervals, where the implied constant is absolute.

\end{lem}
\begin{proof}For each $\left(s,t\right) \in V,$
one gets from \eqref{adj} that
$$|s|^2,|t|^2 \ll 
K_0 \  
\|\text{adj}\left(\Pi\right)\|_{\infty}  |\Delta_Q|^{-1}.$$
Using the estimates $\|\text{adj}\left(\Pi\right)\|_{\infty} \ll \|\Pi\|^2\ll
{\langle Q \rangle}^2
$ concludes the proof of the first assertion.
The norm $\|.\|$ is isometric to the supremum norm 
and hence $V$ is the intersection of the interior of
$3$ plane conic sections. 
Therefore $V$ is a finite union of at most $O(1)$ convex sets, 
where the implied constant is absolute, thus showing that 
$|\partial V|$ is
bounded by an absolute constant multiplied with the length of the box that contains $V$. Our last assertion is a 
consequence of ~\cite{daven2} as the set $V$ is semi--algebraic owing to the 
the fact that $\|.\|$ is isometric to the supremum norm.
\end{proof} 
Define for any $T \in \mathbb{R}_{\geq 1}$ and $n,\sigma,\tau \in \mathbb{N}$ such that 
$\gcd(\sigma,\tau,n)=1$,
\begin{equation}
\label{level0}
M_{\sigma,\tau}\left(T,n\right):=
\#
\left\{\left(s,t\right) \in \mathbb{Z}^2:
\begin{array}{l}
\left(s,t\right) \equiv \left(\sigma,\tau\right) \mod n,\\
t>0,\|\mathbf{q}\left(s,t\right)\|\leq T
\end{array}
\right\}.
\end{equation}
\begin{lem}
\label{lem6} 
For any $T,n,\sigma,\tau$ as above
with
$\gcd(\sigma,\tau,n)=1$ and  
$n|\mathbf{q}\left(\sigma,\tau\right),$
one has
$$
M^*_{\sigma,\tau}\left(T,n\right)=
\sum_{\substack{1 \leq m \leq 
\left(2TK_0 / n\right)^{\frac{1}{2}} \\ \gcd(m,n)=1}}\mu\left(m\right)
M_{\bar{m}\sigma,\bar{m}\tau}
\left(\frac{T}{m^2},n\right),
$$
where $\bar{m}$ denotes 
the inverse of $m\mod{n}.$
\end{lem}
\begin{proof}
The condition $\|\mathbf{q}\left(s,t\right)\|\leq T$
implies by Lemma~\ref{lem5}, that the number of $\left(s,t\right)$ 
counted by $M^*_{\sigma,\tau}\left(T,n\right)$ is finite. Therefore
Lemma~\ref{lem2} may be applied to yield
\begin{equation}
\label{coltrane}
M^*_{\sigma,\tau}\left(T,n\right)=
\sum_{\substack{m=1 \\ \gcd(m,n)=1}}^{\infty}\mu\left(m\right)
M_{\bar{m}\sigma,\bar{m}\tau}\left(\frac{T}{m^2},n\right).
\end{equation}
If $m>\left(2 K_0 T / n\right)^{\frac{1}{2}},$ then each $\left(s,t\right)$ taken into account by 
$M_{\bar{m}\sigma,\bar{m}\tau}\left(\frac{T}{m^2},n\right),$
satisfies $\|\mathbf{q}\left(s,t\right)\|_{\infty} <\frac{n}{2},$
due to \eqref{psy}.
The assumptions on $\sigma,\tau,n,$
imply that $n|\mathbf{q}\left(s,t\right)$
which is only possible if $\mathbf{q}\left(s,t\right)=\mathbf{0}.$
Due to \eqref{adj}, one has $t=0$
which contradicts the definition of
\eqref{level0}. This shows that only integers
$m\leq\left(2 K_0 T / n\right)^{\frac{1}{2}}$
make a non--zero contribution to
\eqref{coltrane}, which concludes the proof of the lemma.
\end{proof} 
Recall the definitions \eqref{psy00} and \eqref{level0}.
\begin{lem}
\label{lem7} For any $T,n,\sigma,\tau$ as above, we have
$$
M_{\sigma,\tau}\left(T,n\right)= \mathrm{vol}\left(V\right)
 \frac{T}{n^2}+O\left(1+ \frac{\left(K_0 T\right)^{\frac{1}{2}}}{n}
\frac{{\langle Q \rangle}}{|\Delta_Q|^{\frac{1}{2}}} \right).
$$
\end{lem}
\begin{proof} 
The quantity 
$M_{\sigma,\tau}\left(T,n\right)$
equals the number of integer points
in the region  
$$
\frac{\ T^{\frac{1}{2}}}{n}V-\left(\frac{\sigma}{n},
\frac{\tau}{n}\right),$$
where $V$ is defined in \eqref{psy00}.
We thus deduce that
\[
M_{\sigma,\tau}\left(T,n\right)
=\sharp
\bigg
\{
\mathbb{Z}^2
\cap
V
\frac{\ T^{\frac{1}{2}}}{n}
\bigg\}
+O\bigg(1+|\partial V|
\frac{\ T^{\frac{1}{2}}}{n}
\bigg),
\]
where $|\partial V|$ denotes
the length of the boundary 
of $V.$ 
The assumptions of the theorem in~\cite[pg.180]{daven1} are fulfilled due to
Lemma~\ref{lem5},
thus yielding
\[
\sharp
\bigg
\{
\mathbb{Z}^2
\cap
V
\frac{\ T^{\frac{1}{2}}}{n}
\bigg\}
=
\mathrm{vol}\left(V\right)
 \frac{T}{n^2}+O\left(1+ \frac{\left(K_0 T\right)^{\frac{1}{2}}}{n}
\frac{{\langle Q \rangle}}{|\Delta_Q|^{\frac{1}{2}}} \right).
\]
This estimate, when combined with the second assertion of  
Lemma~\ref{lem5},
finishes the proof.
\end{proof}

\section{The asymptotic formula}
\label{5s}
We are now in possession of the 
required lemmata to show the validity of Proposition~\ref{prop1}.
Before proceeding
to the proof we should remark that
we shall show the asymptotic formula
of Proposition~\ref{prop1}
with a different constant 
in place of $c_{Q},$
and at the end of this section
we will explain why the two constants coincide.

Let us now define the new constant,
which we denote by $c'_{Q}.$
Recall the definitions
\eqref{ro} and \eqref{psy00}.
Let
\[
\sigma'_{\infty}:=\mathrm{vol}\left(V\right)
\]
and for any prime $p,$
let
\[
\sigma'_p:=\left(1-\frac{1}{p^2}\right)\left(1+\frac{1}{\left(1+\frac{1}{p}\right)} 
\sum_{d \geq 1}
\frac{\rho^*\left(p^d\right)}{p^d}\right).
\]
Lemma~\ref{lem1} shows that the product 
$\prod_{p}\sigma'_p$ taken over all primes $p$ converges and we may thus define
$$
c'_{Q}:=
\sigma'_{\infty}
\prod_{p}\sigma'_p \ .
$$
Notice that Lemma~\ref{lem7}
implies that 
\begin{equation}
\label{vol0}
\sigma'_{\infty} \ll {\langle Q \rangle}^2 \frac{K_0}{|\Delta_Q|}.
\end{equation}

In light of Lemma~\ref{lem3}, it suffices to prove
Proposition~\ref{prop1}
for $\mathcal{N}\left(Q,B\right)$
in place of
$N\left(Q,B\right).$ Combining Lemma~\ref{lem4} and Lemma~\ref{lem6}, gives 
\begin{equation}
\label{miles}
\mathcal{N}\left(Q,B\right)
=
\sum_{k\l| 
\Delta_Q / \gcd(b,e) } 
\mu\left(k\right)
\Osum_{\substack{\left(\sigma,\tau\right) \mod{k\l} \\ k\l|\mathbf{q}\left(\sigma,\tau\right) }}
\sum_{\substack{m \leq 
\left(2BK_0 /k\right)^{\frac{1}{2}} \\ \gcd(m,k\lambda)=1}}\mu\left(m\right)
M_{\bar{m}\sigma,\bar{m}\tau}\left(\frac{B \lambda}{m^2},k \lambda \right).
\end{equation}
Now notice that
for $$\mathcal{L}:= \frac{\left(K_0 B\right)^{\frac{1}{2}}}{k \lambda^{\frac{1}{2}}}
\frac{{\langle Q \rangle}}{|\Delta_Q|^{\frac{1}{2}}},$$
the bound
\eqref{vol0} and Lemma~\ref{lem7}
imply that
$$M_{\bar{m}\sigma,\bar{m}\tau}
\left(\frac{B \lambda}{m^2},k \lambda\right)=
\begin{cases} \sigma'_{\infty}
 \frac{B}{m^2 k^2\lambda}+O\left(\frac{\mathcal{L}}{m} \right) &\mbox{if } m \leq \mathcal{L} \\
O\left(1\right) & \mbox{otherwise. } \end{cases}
$$
The contribution to \eqref{miles}
coming from those $m$ with $m>\mathcal{L}$ is therefore
$\ll_{\epsilon}
\left(BK_0\right)^{\frac{1}{2}}
|\Delta_Q|
{\langle Q \rangle}^{\epsilon}.
$
We have used the bound $\tau_{k}\left(n\right)\ll_{k,\epsilon}n^{\epsilon}$
as well as part $\left(ii\right)$ of 
Lemma~\ref{lem1}.
The contribution of the remaining $m$ is 
\begin{align*}
&\sigma'_{\infty} B \sum_{k\l|\Delta_Q} \frac{\mu\left(k\right)\rho^*\left(k\l\right)}{k^2\l}
\sum_{\substack{  m \leq \mathcal{L}  \nonumber \\
\gcd(m,k\l)=1}} \frac{\mu\left(m\right)}{m^2}\\
&+O_{\epsilon}\left(\left(B K_0\right)^{\frac{1}{2}} \left(\log B K_0\right) 
{\langle Q \rangle}^{1+\epsilon} \gcd(b,e)^{\frac{1}{2}} \right).
\end{align*}
Extending the summation over $m$ to
infinity, the error introduced
in the 
main term is $
\ll_{\epsilon} \left(BK_0\right)^{\frac{1}{2}}
{\langle Q \rangle}^{1+\epsilon}
\gcd(b,e)^{\frac{1}{2}}
,
$
where
we have made use of \eqref{vol0}.
The fact that 
$
\gcd(b,e)^2|\Delta_Q$
and
$\gcd(b,e)|\delta_Q
$
provides the error term in Proposition~\ref{prop1}.
Using the fact that $\rho^*$ is multiplicative and supported on the divisors of $\Delta_Q$ 
we deduce that
$$
\sum_{k,\l \in \mathbb{N}} \frac{\mu\left(k\right)\rho^*\left(k\l\right)}{k^2\l}
\sum_{\substack{ m \in \mathbb{N} \nonumber \\
\gcd(m,k\l)=1}} \frac{\mu\left(m\right)}{m^2}=
\prod_{p}
\left(1-\frac{1}{p^2}+\left(1-\frac{1}{p}\right)\sum_{d \in \mathbb{N}}\frac{\rho^*\left(p^d\right)}{p^d}\right),
$$
which shows that the leading constant is equal to 
$c'_Q$, as desired. 

We proceed to explain why the leading constants
$c_Q$ and $c'_Q$ are equal. One can indeed 
produce an elementary, yet lengthy, argument of this assertion,
performing a parametrisation argument 
over
$\mathbb{Z} / p^n\mathbb{Z}$
for appropriately chosen primes $p$ and positive integers $n,$
instead of over $\mathbb{Q}.$ However,
as the referee kindly pointed out, it is shown in \cite[Sections 3 and 6.2]{peyre}
that $c_Q=c'_Q$ follows from \cite{fmt}.
More precisely, the fact that points are equidistributed on the projective line
implies
that the leading constants agree for any height,
including the one coming from the embedding of the projective line as a conic.
This concludes the proof of Proposition~\ref{prop1}.

\section{The proof of Theorem~\ref{thrm01}}
\label{7s}
In this section we complete the 
proof of Theorem~\ref{thrm01} by transforming the general form $Q$ into one to
which Proposition~\ref{prop1} applies. The 
next lemma shows that one can find 
a suitable transformation with the lowest possible height. 
\begin{lem}
\label{lem10}
Let $\mathbf{a} \in \mathbb{Z}^3_{\text{prim}}.$
Then there exists 
$M \in SL_3\left(\mathbb{Z}\right)$ whose second 
column is $\mathbf{a}$ and whose entries 
have maximum modulus $O\left(\|\mathbf{a}\|_{\infty}\right).$
\end{lem}
\begin{proof}
By renaming indices if needed, 
we may assume that 
$$0 <|a_1| \leq |a_2|\leq |a_3|.$$
Let us notice that an integer solution to the equation
$\mathbf{a}^t \mathbf{y}=1$
exists, owing to the coprimality of 
$\mathbf{a}.$ The previous inequality implies that
we can pick $s,t \in \mathbb{Z}$
such that $\max\{|y_3-a_1 t|,|y_2-a_1 s|\}\leq \frac{|a_1|}{2}.$
Then the integer vector 
$$\mathbf{x}:=\mathbf{y}
+s\left(a_2,-a_1,0\right)
+t\left(a_3,0,-a_1\right)$$
satisfies $\mathbf{a}^t \mathbf{x}=1$
and $\|\mathbf{x}\|_{\infty} \ll \|\mathbf{a}\|_{\infty}.$

We now let $x'_i:=\frac{x_i}{\gcd(x_1, x_2)}, i=1,2$
so that $\gcd(x'_1,x'_2)=1.$ We know therefore that
an integer solution $(x,y)$ of 
$ x'_1 x+x'_2 y=x_3 $ can be found. Considering $y-tx'_1$ in place of $y$ if needed, we can prove as previously that we can find $(x,y)$ 
that satisfy the previous equationin addition to $\max\{|x|,|y|\}\ll
 \|\mathbf{x}\|_{\infty}.$ A direct calculation may then 
reveal that the matrix
$$M:=
 \begin{pmatrix}
 x'_2  & a_1   &  -x\\
 -x'_1 & a_2   &  -y \\
 0     & a_3   &  \gcd(x_1,x_2)
 \end{pmatrix}
$$
possesses the required properties.
\end{proof} 
\textit{Proof of Theorem~\ref{thrm01}.} It is given that
the quadratic form $Q$ possesses a rational zero. One can therefore 
find, using Cassels~\cite{cas}, a non--trivial integer zero $\boldsymbol{\xi}:=\left(x_0,y_0,z_0\right) \in 
\mathbb{Z}^3_{\text{prim}}$ of $Q$
such that $\|\boldsymbol{\xi}\|_{\infty} \ll {\langle Q \rangle}.$
We now transform the form $Q$ using $\mathbf{a}=\boldsymbol{\xi} $ in the previous lemma. It provides
an integer matrix $M$ of determinant $1$
and of size 
\begin{equation}
\label{ccaass}
\|M\|_{\infty} \ll {\langle Q \rangle}
\end{equation} 
such that the quadratic form $Q'$ defined by 
$$Q'\left(\mathbf{x}\right):=Q\left(M \mathbf{x}\right),$$
possesses the zero $\left(0,1,0\right).$
We define the norm given by
$$\|\mathbf{x}\|':=\|M\mathbf{x}\|$$
and notice that
$$
{\langle Q' \rangle}
\ll{\langle Q \rangle}^3.
$$
The fact that $M$ is unimodular implies that
the integer vector $\mathbf{x}$ is primitive
if and only if $M\mathbf{x}$ is. It therefore follows that
$$N\left(Q,B\right)=
N'\left(Q',B\right),
$$
where the notation $N'$ indicates 
a use of the norm $\|.\|'.$ 
Recall the definition \eqref{psy} of $K_0.$
Using the inequality
$\|M^{-1}\|_{\infty}
\leq 2 \|M\|^2_{\infty}$ 
and writing 
\newline
$\mathbf{x}=M^{-1}(M\mathbf{x})$
for all $\mathbf{x}\neq \mathbf{0}$,
implies that
$$
\|\mathbf{x}\|_{\infty}
\leq
2 \|M\|^2_{\infty}K_0\|\mathbf{x}\|'.$$
Therefore \eqref{ccaass} shows that for
$K'_0:=1+\sup_{\mathbf{x}\neq 0}
\frac{\ \|\mathbf{x}\|_{\infty}}{\|\mathbf{x}\|'},$
we have $$K'_0 \ll K_0 {\langle Q \rangle}^2.$$
Finally, notice that 
the discriminants $\Delta_Q, \Delta_{Q'}$ as well as the 
greatest common divisors $\delta_Q$ and  $\delta_{Q'}$ of the $2\times 2$ minors of the matrices 
of the quadratic forms $Q$ and $Q'$
remain invariant
under the unimodular transformation $M.$

We are now in a position to apply Proposition~\ref{prop1} to the form $Q'$
with all involved quantities modified as indicated hitherto.
We are provided with the error term 
\begin{equation}
\label{lastend}
\ll_{\epsilon} \left(B K_0\right)^{\frac{1}{2}} \log\left(B K_0\right) 
\min\left\{|\Delta_Q|^{\frac{1}{4}},\delta_Q^{\frac{1}{2}}\right\}
{\langle Q \rangle}^{4+\epsilon}.
\end{equation}
The bound
$|\Delta_Q| \ll {\langle Q \rangle}^3$
implies that this is
$$
\ll_{\epsilon} \left(B K_0\right)^{\frac{1}{2}} \log\left(B K_0\right) 
{\langle Q \rangle}^{^{\frac{19}{4}}+\epsilon}
$$
so that 
using the value
$\epsilon=\frac{1}{4}$
we obtain
the error term appearing in Theorem~\ref{thrm01}.
Recall the definition \eqref{si}
and \eqref{s} of the local densities.
It remains to show that they satisfy
$$\sigma_{\infty}\left(Q',\|.\|'\right) = \sigma_{\infty}\left(Q,\|.\|\right)$$
and 
$$\ \sigma_p\left(Q'\right) = \sigma_p\left(Q\right)$$
for any prime $p.$ The fact that the matrix $M$ is invertible$\mod{p^n}$ shows that $N_Q^*(p^n)=N_{Q'}^*(p^n)$ is valid,
which when used in \eqref{s} proves the latter equality.
The former is proved by performing the unimodular linear change of variables $\mathbf{x}=M\mathbf{X}$ in \eqref{si}. Hence
$$
\int_{\substack{|Q\left(\mathbf{x}\right)|\leq \epsilon \\
\|\mathbf{x}\| \leq 1}}1 \mathrm{d}\mathbf{x}
=
\int_{\substack{|Q'\left(\mathbf{X}\right)|\leq \epsilon \\
\|\mathbf{X}\|' \leq 1}}1\,\mathrm{d}\,\mathbf{X},$$
which finishes the proof of Theorem~\ref{thrm01}.

\subsection*{Acknowledgements}
The author would like to express his gratitude to T. Browning for suggesting the problem and for his valuable assistance during the course of the project.
He is furthermore indebted to Dr. Christopher Frei for useful  comments regarding an earlier version of this paper.

\end{document}